\documentclass[a4paper,11pt,english,german]{amsart}
\usepackage[utf8]{inputenc}
\usepackage[left=2cm,right=2cm,
    top=2cm,bottom=2cm,bindingoffset=0cm]{geometry}

\usepackage{amssymb}
\usepackage{amsmath}
\usepackage{amsthm}
\usepackage{mathtools}
\usepackage{multirow}
\usepackage{thmtools}
\usepackage{nicefrac}
\usepackage{comment}
\usepackage{url}
\usepackage[colorlinks=true, citecolor=blue, linkcolor=black]{hyperref}

\usepackage[warn]{textcomp}

\usepackage{enumerate}
\usepackage{mathrsfs}

\usepackage{epsfig,graphicx}

\makeatletter
\ifcsname phantomsection\endcsname
    \newcommand*{\qrr@gobblenexttocentry}[5]{}
\else
    \newcommand*{\qrr@gobblenexttocentry}[4]{}
\fi
\newcommand*{\addsubsection}{%
    \addtocontents{toc}{\protect\qrr@gobblenexttocentry}%
    \subsection}
\makeatother

\DeclareMathOperator{\Cl}{Cl}
\DeclareMathOperator{\Pic}{Pic}
\DeclareMathOperator{\IP}{\mathbb{P}}
\DeclareMathOperator{\Q}{\mathbb{Q}}
\DeclareMathOperator{\C}{\mathbb{C}}
\DeclareMathOperator{\Z}{\mathbb{Z}}
\DeclareMathOperator{\Sing}{Sing}
\DeclareMathOperator{\Bs}{Bs}
\DeclareMathOperator{\Jac}{Jac}
\DeclareMathOperator{\Socle}{Socle}
\DeclareMathOperator{\eps}{\epsilon}
\DeclareMathOperator{\Proj}{Proj}
\DeclareMathOperator{\rk}{rk}

\newtheorem{thm}{Theorem}[section]
\newtheorem{lem}[thm]{Lemma}

\newtheorem{prop}[thm]{Proposition}
\newtheorem{cor}[thm]{Corollary}
\newtheorem{conj}[thm]{Conjecture}
\theoremstyle{definition}

\newtheorem{Def}[thm]{Definition}
\newtheorem{remark}[thm]{Remark}

\title{Ciliberto--Di Gennaro conjecture for sextic hypersurfaces}
\author{Ksenia Kvitko}
\address{HSE University, Russian Federation}
\email{ksenia.kvitko@yahoo.com}

\begin{document}
\begin{abstract}
The Ciliberto--Di Gennaro conjecture addresses the factoriality of three-dimensional nodal hypersurfaces, and their geometric properties. We prove this conjecture for hypersurfaces of degree 6 by adapting a recent technique due to R. Kloosterman.
\end{abstract}

\maketitle

\tableofcontents 

\section{Introduction}

Let $X$ be a hypersurface in $\IP^4$ of degree $d$ that has only ordinary double points (nodes) as singularities. Such a hypersurface is called \emph{nodal}. One could ask whether it is factorial or equivalently $\Q$-factorial. This property has important implications for the study of rationality (see \cite{cheltsov2005birationally}). C.~Ciliberto and V. Di Gennaro formulated in \cite{ciliberto2004factoriality} the following conjecture.

\begin{conj}\label{conj} Let $X\subset\IP^4$ be a nodal hypersurface of degree $d$ with at most $2(d-2)(d-1)$ singular points.
Then one of the following holds:
\begin{enumerate}
	\item[\emph{(1)}] $X$ is factorial;
	\item[\emph{(2)}] $X$ contains a plane and has at least $(d-1)^2$ nodes;
	\item[\emph{(3)}] $X$ contains a quadric surface and has exactly $2(d-2)(d-1)$ nodes.
\end{enumerate}
\end{conj}

Conjecture \ref{conj} was proved for $d=3$ by H.~Finkelnberg and J.~Werner \cite{finkelnberg1989small}, for $d=4$ by I.~Cheltsov and C.~Shramov \cite{cheltsov2006nonrational,shramov2007factorial}, and for $d\geqslant 7$ by R.~Kloosterman in \cite{kloosterman2022maximal}. In addition, it was shown in \cite{cheltsov2010factoriality,cheltsov2010conjecture} that a non-factorial nodal threefold $X\subset \IP^4$ of degree $d$ must have at least $(d-1)^2$ nodes; moreover, if equality holds then $X$ contains a plane.

In this paper we present a proof of the Ciliberto--Di Gennaro conjecture for sextic hypersurfaces using the method from \cite{kloosterman2022maximal}.

\begin{thm}\label{main}
	Conjecture \ref{conj} holds for $d=6$.
\end{thm}

The paper is organized as follows. Section~\ref{2} recalls basic results on the Hilbert function of a graded ring, including Macaulay's and Gotzmann's Theorems. Section~\ref{3} treats artinian Gorenstein rings and describes results on unimodality, which concerns the behavior of the Hilbert function. In Section~\ref{4}, we~consider a homogeneous ideal giving artinian Gorenstein quotient ring and analyze the degrees of its generators. Here, the estimates on the values of the Hilbert function are given. In Section~\ref{5}, we follow an approach from \cite{kloosterman2022maximal} to prove Theorem \ref{main}. We give a notion of defect used to catch factoriality property and deeply connected to the Hilbert function of an ideal of singularities~$\Sing(X)$. Switching to algebraic methods, we split Kloosterman's proof into several parts and fix the part that fails for degree 6. Specifically, let $X$ be a three-dimensional nodal hypersurface of degree $d \geqslant 6$ in $\Proj (R)\simeq \mathbb{P}^4$ and let $J\subset R$ be a homogeneous ideal of $\Sing(X)$. Let~$\{\ell=0\}$ be a general hyperplane not passing through $\Sing(X)$. By taking an ideal $\overline{J}$ corresponding to this hyperplane section and possibly adding other homogeneous polynomials $f_j$ of degree $d_j$, where~$d-1 \leqslant d_j \leqslant 2d-4$, we can construct such an ideal $I\supset \overline{J}$ in $S=R/(\ell)$ that the corresponding quotient ring $S/I$ will be artinian Gorenstein of socle degree $2d-4$. In other words, we obtain some local ring having symmetric Hilbert function $h_I(k)$ for $k\leqslant 2d-4$ and $h_I(k)=0$ for $k\geqslant 2d -3$. Then a lower bound on the number~$\# \Sing(X)$ can be found in terms of the values $h_I(k)$. It turns out that $\Sing(X)$ contains a complete intersection of multidegree either $(1,1,d-1,d-1)$ or $(1,2,d-2,d-1)$ if $h_I(d-4)$ equals at most $2d-7$. Otherwise, Kloosterman showed that $X$ of degree $d\geqslant 7$ has more nodes than it was supposed in Conjecture \ref{conj}. The consequent estimation on $\# \Sing(X)$ through the values of $h_I$ and then through the Hilbert function of complete intersection ideal $I_{CI}$ of multidegree $(1,2,d-2,d-1)$ works perfectly in the case $d>7$. For $d=7$ this lower bound is exactly $2(d-2)(d-1)$. So, there is only one bad option for values $h_I$ (i.e. when all inequalities on $h_I(k)$ become equalities). As the geometric analysis carried out in \cite{kloosterman2022maximal} shows, this option does not realize. However, this analysis is not applicable for sextic hypersurfaces with more than one options giving $\# \Sing(X)$ less than $2(d-2)(d-1)+1$. According to the theorems on unimodality for such an ideal $I$ as we work with, if $h_I(k)\leqslant h_I(k+1)$ for some $k\leqslant d-4$, then $h_I(k+1)\leqslant h_I(k+2)$. Hence, we limit all possible bad sequences of $h_I(k)$ to 2 cases and exclude them too by analyzing the dimension of an intersection of hypersurfaces of fixed degree $t$ passing through the subscheme defined by the $t$-graded part of $I$.

\addsubsection*{Acknowledgements}

	The author would like to thank Constantin~Shramov for proposing the problem, providing valuable feedback on the draft, and engaging in helpful discussions. The author is also grateful to Manolis~C.~Tsakiris for his expertise on artinian Gorenstein rings and to Remke~Kloosterman for insightful comments.

\section{Preliminaries}\label{2}

Given the polynomial ring $S = \C[x_0,\dots,x_n]$ and a homogeneous ideal $I \subset S$, the Hilbert function $h_I(k) = \dim (S/I)_k$ measures the dimension of the $k$-th graded component of $S/I$. The Hilbert polynomial $p_I(t) \in \Q[t]$ agrees with the Hilbert function for large values of $k$.

To understand the Hilbert function, we use a notion of so-called \emph{Macaulay expansion}. Let $d \geqslant 1$ be an integer, and let $h_I(d)=C$ represent the value of the Hilbert function at $d$. The Macaulay $d$-expansion provides a unique representation of $C$ in base $d$
\begin{equation*}
	C=\sum\limits_{i=1}^{d}\binom{i+\eps_i}{i},
\end{equation*}
where the coefficients $\eps_i$ satisfy $\eps_d \geqslant \eps_{d-1} \geqslant \dots \geqslant \eps_{1} \geqslant -1$. This expansion can be obtained inductively as follows. The number $\eps_d$ is the largest integer such that $\binom{d+\eps_d}{d} \leqslant C$. The numbers $\eps_i$ for $i<d$ are the coefficients in the expansion of $C-\binom{d+\eps_d}{d}$ in base $d-1$.

Using the Macaulay expansion of $C$ for any positive integers $d$ we define the  functions $^{\langle d \rangle}:\Z\to \Z$ and $_{\ast d}:\Z\to \Z$:

\begin{equation*}
	C^{\langle d \rangle}=\sum\limits_{i=1}^{d}\binom{i+\eps_i+1}{i+1}, \hspace{10pt} C_{\ast d}=\sum\limits_{i=2}^d \binom{i+\eps_i -1}{i-1}.
\end{equation*}
Note that $C \mapsto C_{\ast d}$ and $C\mapsto C^{\langle d\rangle}$ are increasing functions in $c$.

\begin{Def}
A sequence $\{h_k \, \mid\,k\geqslant 0\}$ of nonnegative integers is called an \emph{$O$-sequence} if 
\begin{equation*}
	h_0=1,\hspace{5pt}\text{and}\hspace{5pt} h_{k+1}\leqslant h_k^{\langle k\rangle}\hspace{5pt}\text{for all }k.
\end{equation*}	
\end{Def}

\begin{thm}[Macaulay \cite{macaulay1927some}] \label{macaulay}
	The following are equivalent:
	\begin{enumerate}
		\item[\emph{(1)}] a sequence $\{h_k \,:\,k\geqslant 0\}$ of nonnegative integers is an $O$-sequence;
		\item[\emph{(2)}] there exists a polynomial ring $S$ and a homogeneous ideal $I\subset S$ such that $h_I(k)=h_k$ \\for all $k$.
	\end{enumerate}
\end{thm}
\begin{cor}\label{macaulay_zero}
If $h_I(N)=0$ for some positive integer $N$ then $h_I(k)=0$ for all $k\geqslant N$.	
\end{cor}
\begin{proof}
	Indeed, by Theorem \ref{macaulay} one has $h_I(N+1)\leqslant 0^{\langle N\rangle}=0$. Hence, $h_I(N+1)=0$ and similarly $h_I(k)=0$ for all $k>N+1$.
\end{proof}

Having that $h_I(d)<h_I(d-1)^{\langle d-1\rangle }$ for the values of the Hilbert function, we can also obtain information on $h_I(d-1)$.

\begin{cor}\label{lower_bound}
	Let $I\subset S$ be a homogeneous ideal, $d\geqslant 2$ an integer and $C=h_I(d)$. Then
	\begin{equation*}
		h_I(d-1)\geqslant C_{\ast d}.
	\end{equation*}
	Moreover, if $\eps_1$ is nonnegative then $h_I(d-1)>h_{\ast d}$ holds.
\end{cor}
\begin{proof}
We have 
$$h_I(d)=\sum\limits_{i=1}^d \binom{i+\eps_i}{i}$$
 and $h_I(d)\leqslant h_I(d-1)^{\langle d-1\rangle}$ from Theorem \ref{macaulay}. So, in terms of the Macaulay expansions the following inequality holds
\begin{equation*}
	\sum\limits_{i=1}^d \binom{i+\eps_i}{i} \leqslant \sum\limits_{i=1}^{d-1}\binom{i+\eps_i'+1}{i+1}.
\end{equation*}
After replacing $i+1$ with $i$, the required condition is obtained
\begin{equation*}
	h_I(d)_{\ast d}=\sum\limits_{i=2}^d \binom{i+\eps_i-1}{i-1} \leqslant \sum\limits_{i=1}^{d-1}\binom{i+\eps_i'}{i}=h_I(d-1).
\end{equation*}	
\end{proof}

\begin{remark}\label{macaulay_examples}
For small $h$ we have the following Macaulay expansions in base $d$.
\begin{itemize}
	\item For $h\leqslant d$ we have $\eps_d=\dots =\eps_{d-h+1}=0$ and $\eps_{d-h}=\dots=\eps_1=-1$. Hence $h^{\langle d\rangle}=h$.
	\item For $d+1\leqslant h \leqslant 2d$ we have $\eps_d=1$, $\eps_{d-1}=\dots =\eps_{d-a}=0$, $\eps_{d-a-1}=\dots = \eps_1=-1$, where $a=h-d-1$. Hence $h^{\langle d \rangle}=h+1$.
	\item For $h=2d+1$ we have $\eps_d=\eps_{d-1}=1$ and all other $\eps_i$ equal $-1$. Hence $h^{\langle d\rangle}=2d+3=h+2$.
\end{itemize}	
\end{remark}

Applying Corollary \ref{lower_bound} repeatedly yields
\begin{cor}\label{corollary_example}
Let $I\subset S$ be a homogeneous ideal, $d\geqslant 2$ an integer and $h=h_I(d)$. For $0\leqslant k \leqslant d$ we have 
\begin{equation*}
	h_I(k)\geqslant 
	\begin{cases}
	\min(h,k+1) \hspace{65pt}\text{if } h \leqslant d;\\
	\min(k+(h-d),2k+1) \hspace{12pt}\text{if } d+1 \leqslant h \leqslant 2d;\\
	2k+1	\hspace{95pt}\text{if } h=2d+1.
	\end{cases}
\end{equation*}	
\end{cor}

The following result will be used to detect the Hilbert polynomial of the ideal generated by $I_d$.

\begin{thm}[Gotzmann \cite{gotzmann1978bedingung}] \label{gotzmann}
	Let $V \subset S_d$ be a subspace in the $d$-graded part of $S$ and $J\subset S$ be the ideal generated by $V$. Set $C=h_J(t)$. If $h_J(d+1)=C^{\langle d \rangle}$ then we have $h_J(k+1)=h_J(k)^{\langle k\rangle}$ for all $k \geqslant d$. In particular, the Hilbert polynomial $p_J(t)$ of $J$ is given by
	\begin{equation*}
		p_J(t)=\sum\limits_{i=1}^{d}\binom{t+ \eps_i}{t},
	\end{equation*} 
	and the dimension of $V(J)$ equals $\eps_d$.
\end{thm}

\section{Structure of artinian Gorenstein rings}\label{3}

Here, we provide the theory for computations in artinian rings (see \cite[Chapter 4]{vasconcelos2004computational} for details).

Let $S$ be the polynomial ring $\C[x_0,\dots,x_n]$ and let $I$ be a homogeneous ideal of $S$ such that $A=S/I$ is an artinian ring. Then there is a decomposition of $A$ into a product of local rings
\begin{equation*}
	A=A_1 \times \dots \times A_r,
\end{equation*}
and the Jacobson radical $\mathfrak{J}(A)$ of $A$ can be written as
\begin{equation*}
	\mathfrak{J}(A)=\mathfrak{m}_1\times \dots \times \mathfrak{m}_r,
\end{equation*}
where $\mathfrak{m}_i$ is the maximal ideal of $A_i$. In the lattice of ideals of $A$, sitting at the opposite end of $\mathfrak{J}(A)$ there lies another distinguished ideal, the socle of $A$:

\begin{Def}
The \emph{socle} of the artinian ring $A$ is the annihilator of its Jacobson radical $\mathfrak{J}(A)$. It will be denoted by $\Socle(A)$.	
\end{Def}

The socle of $A$ can also be presented as a product
\begin{equation*}
	\Socle(A)=(0:\mathfrak{m}_1)\times \dots \times (0:\mathfrak{m}_r).
\end{equation*} 
\begin{Def}[{\cite[Proposition 21.5]{eisenbud2013commutative}}]
An artinian ring $A$ over $\C$ is called \emph{Gorenstein} if one of the following equivalent conditions holds:
\begin{enumerate}
\item	each $A$-module $(0:\mathfrak{m}_i)$ is generated by one element $s_i$;
\item each $A_i$-module $(0:\mathfrak{m}_i)$ is simple.
\end{enumerate}
\end{Def}

\begin{remark}\label{1-dimensional}
 Simpleness of $M=(0:\mathfrak{m}_i)$ implies $M\simeq A_i/\mathfrak{m}_i$ (see \cite[Theorem 2.13]{eisenbud2013commutative}). Hence, each annihilator of $\mathfrak{m}_i$ is a 1-dimensional vector space over $\C$.	
\end{remark}

Artinian Gorenstein rings have a wonderful property which we will refer to as Gorenstein duality.

\begin{thm}[{\cite[Theorem 9.1]{huneke1999hyman}}]\label{gorenstein-duality}
Let $A$ be a $0$-dimensional Gorenstein ring $S/I$. Suppose that
$$A=A_0\oplus \dots \oplus A_e,$$
where $A_i$, $i=0,\dots, e$, is graded component, $A_0=\C$, and $A_e\neq 0$. Then
the pairing 
$$A_k \times A_{e-k}\to A_e\simeq \mathbb{C}$$
is perfect and identifies $A_k$ with the dual of $A_{e-k}$. In particular,
 the ideal $I$ has symmetric Hilbert function
\begin{equation*}
h_I(k)=h_I(e-k)\hspace{15pt}\text{for all }k=0,\dots,e.
\end{equation*}	
\end{thm}

\begin{Def}
	Let $A$ be a local artinian Gorenstein ring and $s$ be a polynomial that generates its socle. Then the degree of the polynomial $s$ is called the \emph{socle degree}.
\end{Def}

There are special cases when the socle and the socle degree can be found explicitly. We describe below rings arising from complete intersections of $n+1$ hypersurfaces in $\IP^n$. A homogeneous ideal~$I\subset S$ is called a \emph{complete intersection ideal} if it is generated by a regular sequence $(f_0,\dots, f_n)$ of homogeneous polynomials $f_i$. Let $d_i=\deg f_i$. Then the multidegree of complete intersection ideal is the sequence $(d_0,\dots, d_n)$. The quotient ring~$A=S/I$ is called a \emph{complete intersection ring}.

We use a result of Tate \cite[Theorem A.3(4)]{mazur1970local} (see also \cite{storch1975spurfunktionen}) that gives a concrete description~of the socle of $A$.

\begin{thm}[Tate]\label{tate}
Let $I=(f_0,\dots, f_n) \subset S$ be a homogeneous ideal such that $A=S/I$ is both local and a complete intersection ring. Then $\Socle(A)$ is generated by the  determinant $\det (\Jac)$ of the Jacobian matrix of polynomials $f_0,\dots, f_n$.
\end{thm}
\begin{proof}
	First, we show that $\det (\Jac)$ annihilates the maximal ideal $\mathfrak{m}=(x_0,\dots,x_n)$. For each $f_i$ we~have 
	\begin{equation*}
		\sum\limits_{j=0}^n \frac{\partial f_i}{\partial x_j} \cdot x_j=(\text{deg }f_i)\cdot f_i.
	\end{equation*}
	From this it follows that $\det (\Jac)$ annihilates $x_k$ for all $k$:
	\begin{multline*}		
	\text{det}\bigg(\frac{\partial f_i}{\partial x_j}\bigg)\cdot x_k =\text{det}\bigg(\frac{\partial f_i}{\partial x_1} \hspace{10pt} \dots \hspace{10pt} \frac{\partial f_i}{\partial x_k}\cdot x_k \hspace{10pt}\dots\hspace{10pt}\frac{\partial f_i}{\partial x_n}\bigg)=\\
	=\text{det}\bigg(\frac{\partial f_i}{\partial x_1} \hspace{10pt} \dots \hspace{10pt} (\text{deg }f_i)\cdot f_i -\sum\limits_{j\neq k}\frac{\partial f_i}{\partial x_j}\cdot x_j \hspace{10pt}\dots\hspace{10pt}\frac{\partial f_i}{\partial x_n}\bigg)=0.
	\end{multline*}
	
	It remains to show that 	$\det (\Jac)$ is the generator of the socle. If $\det (\Jac)=h\cdot \varphi$ where $\varphi$ is the generator and $h$ is not a unit, then $h$ lies in the maximal ideal $\mathfrak{m}$. Therefore, $\det (\Jac)$ should be equal to 0. Contradiction. 
\end{proof}
\begin{cor}\label{socle_degree}
	Let $I=(f_0,\dots,f_n)\subset S$ be a complete intersection ideal, and let $d_i=\deg f_i$. Then the socle degree of $S/I$ equals $$\sum\limits_{i=0}^n (d_i -1).$$
\end{cor}
\begin{cor} \label{complete_intersection_coincides_with_ideal}
	Let $I\subset S$ be a proper homogeneous ideal and let $I\supset I_{CI}=(f_0,\dots,f_n)$ be a complete intersection ideal. Suppose both rings $S/I$ and $S/I_{CI}$ are artinian Gorenstein and of the same socle degree. Then $I=I_{CI}$.
\end{cor}
\begin{proof}
Let $N$ and $N_{CI}$ denote the~socle degree of $S/I$ and $S/I_{CI}$. Since $I_{CI}\subset I$, we have $N\leqslant N_{CI}$.
 Then we can use the observation from \cite{otwinowska2002} (at the end of Section 1). It implies there exists a~homogeneous polynomial $F\in S$ of degree $N_{CI}-N$ such that $I=(I_{CI}:F)$. By assumption, the~socle degrees $N$ and $N_{CI}$ coicides. So, $F$ is a nonzero constant and $I=I_{CI}$.
\end{proof}

Now, consider the sequence of non-zero values of the Hilbert function $h_I(k)=\dim (S/I)_k$, where $S/I$ is artinian Gorenstein of socle degree $e$. It is also called an $h$\emph{-vector} and denoted by $(h_0,h_1,\dots,h_e)$, where $h_0=1$ and $h_e=1$.

\begin{Def}
	The $h$-vector $(h_0,h_1,\dots,h_e)$ is said to be \emph{unimodal} if it is never strictly increasing after a strict decrease, i.e., the following condition is satisfied for some $1\leqslant k \leqslant e$
	\begin{equation*}
		h_0 \leqslant h_1 \leqslant \dots \leqslant h_{k-1} \leqslant h_k \leqslant h_{k+1} \geqslant \dots \geqslant h_e.
	\end{equation*}
\end{Def}

It is known that all $h$-vectors with $h_1\leqslant 4$ and $h_4\leqslant 33$ are unimodal (see \cite[Theorem 3.1]{migliore2007characterization}). However, in case of $h_1\leqslant 3$ there is even stronger restriction.

\begin{thm}[{\cite[Theorem 4.2]{stanleyhilbert}}]\label{stanley}
	Let $(h_0,h_1,h_2,\dots,h_e)$ be a sequence of nonnegative integers with $h_1\leqslant 3$ and $h_e\neq 0$. Then it is a $h$-vector of a Gorenstein ring if and only if the following two conditions are satisfied:
	\begin{enumerate}
		\item[\emph{(1)}] $h_i=h_{e-i}$ \hspace{10pt} for $0\leqslant i \leqslant e$;
		\item[\emph{(2)}] $(h_0,h_1 -h_0, h_2-h_1,\dots, h_t -h_{t-1})$ is an $O$-sequence, where $t=[\frac{e}{2}]$.
	\end{enumerate}
\end{thm}
\begin{cor}\label{136676631}
	There is no artinian Gorenstein ring with $h$-vector 
	$$(1,3,6,6,7,6,6,3,1) \hspace{10pt}\text{or}\hspace{10pt}(1,3,6,6,8,6,6,3,1).$$
\end{cor}
\begin{proof}
	The sequence of the differences $(h_0,h_1-h_0,h_2-h_1,h_3-h_2,h_4-h_3)$ is $(1,2,3,0,1)$ or $(1,2,3,0,2)$. Note that $$2>1>0^{\langle 4 \rangle}=0.$$
	Therefore, $(1,2,3,0,1)$ and $(1,2,3,0,2)$ are not $O$-sequences by Theorem \ref{macaulay}. Consequently, there is no Gorenstein ring with $h$-vector $$(1,3,6,6,7,6,6,3,1)\hspace{10pt}\text{or}\hspace{10pt}(1,3,6,6,8,6,6,3,1)$$ by Theorem \ref{stanley}.
\end{proof}

\section{Estimates for values of Hilbert function}\label{4}

Recall that the $t$-graded piece of $S=\C[x_0,\dots ,x_n]$ corresponds to a set of hypersurfaces of degree~$t$ in~$\IP^n$. Similarly, given a homogeneous ideal $I=(g_1,\dots,g_m)$, one could think of a set $|I_t|$ of hypersurfaces of degree $t$ passing through the subscheme $Z$ defined by the generators $g_i$ with $\deg g_i\leqslant t$. We will call such sets $|I_t|$ linear systems, and by the base locus $\Bs\, |I_t|$ of $|I_t|$ we will mean the~subscheme~$Z$.

Below we list a technical result of Kloosterman that will be used in the next section to analyze the structure of ideal $I$ having many generators of small degree.

 \begin{lem}[{\cite[Lemma 5.2]{kloosterman2022maximal}}] \label{technical_lemma}
 Assume $I \subset S$ be a homogeneous ideal and	$S/I$ is an artinian Gorenstein ring of socle degree $N$. Let $d_k=\min\{t\in \Z \, \mid \,\dim\Bs\, |I_t|\leqslant k\}$. Then 
 \begin{equation*}
 	\sum\limits_{k=-1}^{n-1}d_k\geqslant N+n+1.
 \end{equation*}
 \end{lem}
 \begin{remark}\label{ci_ideal_inside_ag_ideal}
 	By definition of $d_k$ in Lemma \ref{technical_lemma}, it follows that these numbers capture those moments when the dimension of $I_t$ changes. Therefore, the ideal $I$ is generated by at least $n+1$ polynomials $g_0,\dots,g_n$ having $\deg g_i=d_{n-1-i}$ for $i=0,\dots,n$. In other words, $I$ contains a complete intersection ideal of multidegree $(d_{n-1},\dots,d_{-1})$.
 \end{remark}

\begin{prop}\label{new}
Let $I$ be a homogeneous ideal in $S$. Suppose it has $h_I(k)=h_I(k-1)=h_I(k-2)$ and $\Bs|I_k|=\varnothing$. Then $\Bs|I_{k-1}|=\varnothing$.
	\end{prop}
\begin{proof}
Suppose that $\Bs|I_{k-1}|\neq \varnothing$. Then the ideal $\widehat{I}$ generated by the subspaces $I_j \subset S_j$, where $j=1,\dots,k-1$, 
defines a non-empty subscheme $Z$ in $\Proj\, S$ and has the same values of Hilbert function as $I$ up to degree~$k-1$, i.e., $h_{\widehat{I}}(t)=h_I(t)$ for all $t\leqslant k-1$. A general hyperplane section of $Z$ by $\{\ell=0\}$ gives a short exact sequence
\begin{equation}\tag{4.1}
	0\to \big(S/\widehat{I}\big)_{t-1} \xrightarrow{\cdot \ell} \big(S/\widehat{I}\big)_{t}\to \big(S/(\widehat{I},\ell)\big)_{t}\to 0 \label{short_exact_sequence}
\end{equation}
for any integer $t\geqslant 1$. So, the ideal $(\widehat{I},\ell)$ has the Hilbert function 
\begin{equation*}
	h_{(\widehat{I},\ell)}(t)=h_{\widehat{I}}(t)-h_{\widehat{I}}(t-1)
\end{equation*}
for $t>0$ and $h_{(\widehat{I},\ell)}(0)=1$. Therefore, we have 
$$h_{(\widehat{I},\ell)}(k-1)=0$$
by assumption and $h_{(\widehat{I},\ell)}(t)=0$ for all $t>k-1$ by Theorem \ref{macaulay}. Consequently, $I_{k}$ and $\widehat{I}_k$ coincide, and so do the base loci, which leads to a contradiction.
\end{proof}

At the rest of the section, we obtain a few restrictions on the Hilbert function of a certain artinian Gorenstein ring that will help to estimate the numbers of singularities.

 \begin{prop} \label{old}
 	Given an artinian Gorenstein ring $S/I$ of socle degree $2d-4$, $d\geqslant 6$, assume $h_I(d-4)>2d-7$ and $\Bs|I_{d-1}|=\varnothing$. Then $h_I(k)>2k+1$ and $h_I(2d-4-k)>2k+1$ for $2\leqslant k \leqslant d-4$ and $h_I(k)\geqslant 2d-6$ for $d-3 \leqslant k \leqslant d-1$.
 \end{prop}
\begin{proof}
For the reader's convenience	 we reproduce the proof from \cite[Lemma 5.4]{kloosterman2022maximal}. 

By assumption, $h_I(d-4)> 2(d-4)+1$. From Remark \ref{macaulay_examples} the Macaulay expansion of $2d-7$ in base $d-4$ is
$$2d-7 = \sum\limits_{i=1}^{d-4}\binom{i+\eps_i}{i}=\binom{d-3}{d-4}+\binom{d-4}{d-5}+\binom{d-7}{d-6}+\dots+\binom{0}{1}.$$
So, we have 
\begin{multline*}
	(2d-7)_{\ast d-4}=\sum\limits_{i=2}^{d-4}\binom{i+\eps_i-1}{i-1}
	=\binom{d-4}{d-5}+\binom{d-5}{d-6}+\binom{d-8}{d-7}+\dots+\binom{0}{1}=2d-9.
\end{multline*}

Applying Corollary \ref{lower_bound}, we get
$$h_I(d-5)\geqslant \big(h_I(d-4)\big)_{\ast d-4}> (2d-7)_{\ast d-4}=2d-9=2(d-5)+1 =2(d-5)+1.$$ 
Then $h_I(d-5)>2(d-5)+1$ and using induction one can show $h_I(k)>2k+1$ for $2\leqslant k \leqslant d-4$. Indeed, we write the Macaulay expansion of $2k+1$ in base $k$
$$
2k+1=\sum\limits_{i=1}^k \binom{i+\eps_i}{i}=\binom{k+1}{k}+\binom{k}{k-1}+\binom{k-3}{k-2}+\dots+\binom{0}{1}
$$
and obtain
$$
h_I(k-1)\geqslant ( h_I(k) )_{\ast k}>(2k+1)_{\ast k}=\binom{k}{k-1}+\binom{k-1}{k-2}+0+\dots+0=2(k-1)+1.
$$
Furthermore, application of Gorenstein duality to $h_I(k)$ for  $2\leqslant k \leqslant d-4$ gives the inequality 
$$
h_I(k)=h_I(2d-4-k)>2k+1.$$

 Since $h_I(d)\geqslant 2d-6$, the Macaulay expansion of $2d-6$ in base $d$ is either
 $$2d-6=d=\binom{d}{d}+\binom{d-1}{d-1}+\dots +\binom{1}{1}$$
 if $d=6$ or
 $$ 2d-6=\binom{d+1}{d}+\binom{d-1}{d-1}+\dots+\binom{7}{7}+\binom{5}{6}+\dots+\binom{0}{1}$$
 if $d>6$. So, we have 
 $$(2d-6)_{\ast d}=\binom{d-1}{d-1}+\binom{d-2}{d-2}+\dots+\binom{1}{1}=2d-7,$$
 $$(2d-6)_{\ast d}=\binom{d}{d-1}+\binom{d-2}{d-2}+\dots+\binom{6}{6}+\binom{4}{5}+\binom{0}{1}=2d-7$$
 for $d=6$ and $d>6$, respectively. Then Corollary \ref{lower_bound} implies that $h_I(d-1)\geqslant 2d-7$. If $h_I(d-1)$ equals $2d-7$, then 
 $$
2d-7=\binom{d}{d-1}+\binom{d-2}{d-2}+\dots+\binom{6}{6}+\binom{4}{5}+\dots+\binom{0}{1}$$
 and due to Theorem \ref{macaulay} we have
 $$h_I(d)\leqslant(h_I(d-1))^{\langle d-1\rangle}=\binom{d+1}{d}+\binom{d-1}{d-1}+\dots+\binom{7}{7}+\binom{5}{6}+\dots+\binom{1}{2}=2d-6.$$
 Hence, $h_I(d)=2d-6$ and from Theorem \ref{gotzmann} the Hilbert polynomial of the ideal $(I_d)$ equals $k+(d-6)$, where $(I_d)$ is generated by the subspace $I_d$. However, by assumption, the base locus of $|I_{d-1}|$ is empty and the base locus of $|I_d|$ is empty too. Therefore, we have $h_I(d-1)\geqslant 2d-6$. Similarly, 
 	for $2d-6$ the Macaulay expansion in base $d-1$ is
 	$$2d-6=\binom{d}{d-1}+\binom{d-2}{d-2}+\dots+\binom{5}{5}+\binom{3}{4}+\dots+\binom{0}{1},$$
 	and $h_I(d-2)\geqslant (2d-6)_{\ast d-1}=2d-7$. Suppose $h_I(d-2)=2d-7$. Then the base locus of $I_{d-1}$ contains a 1-dimensional component, which contradicts the assumption that $\Bs|I_{d-1}|$. So, we have $h_I(d-2)\geqslant 2d-6$.
\end{proof}

\begin{lem}\label{auxiliary_cii}
		Given an artinian Gorenstein ring $S/I$ of socle degree $2d-4$, $d\geqslant 6$, assume the following conditions are satisfied:
		\begin{enumerate}
		\item[\emph{(1)}] $h_I(1)\geqslant 3$;
		\item[\emph{(2)}] $h_I(k)=h_I(2d-4-k)\geqslant 2k+2$ for $2\leqslant k \leqslant d-4$;
		\item[\emph{(3)}]  $h_I(k)\geqslant 2d-6$ for $d-3\leqslant k \leqslant d-1$.
		\end{enumerate}
		Then the inequality holds
		\begin{equation}\label{inequality_for_singularities}\tag{4.2}
			 \sum\limits_{k=0}^{2d-4} h_I(k)>2(d-2)(d-1)
		\end{equation}
		except when $d= 7$ and $I$ has $h$-vector
		$$(1,3,6,8,8,8,8,8,6,3,1)$$
		or when $d=6$ and $h$-vector of $I$ is
		$$(1,3,6,6,6,6,6,3,1),\hspace{10pt} (1,3,6,6,7,6,6,3,1), \hspace{10pt}(1,3,6,7,6,7,6,3,1), \hspace{10pt} (1,3,7,6,6,6,7,3,1), $$
		 $$(1,4,6,6,6,6,6,4,1), \hspace{10pt} (1,3,6,6,8,6,6,3,1).$$
\end{lem}

\begin{proof}
	In order to prove inequality \eqref{inequality_for_singularities}, consider an auxiliary number sequence $h'(k)$, $k=0,\dots 2d-4$. Let $h'(0)=1$, $h'(1)=3$, $h'(k)=2k+1$ for $2\leqslant k \leqslant d-3$, and $h'(d-2)=2(d-2)$.

	Using conditions (1)-(3), we compare the values $h_I(k)$ and $h'(k)$ in Table \ref{table}.
	
	\begin{table}[h!]
	\centering
		\begin{tabular}{| c | c | c | c | c | c | c | c | c | c | c |}\hline 
		$k$ & 0 & 1 & 2  & \dots & $m$ & \dots & $d-4$ & $d-3$ & $d-2$ & \dots \\ \hline
		$h_I(k)$ & 1 & $\geqslant 3$ & $\geqslant 6$ & \dots & $\geqslant 2m+2$ & \dots & $\geqslant 2d-6$ &  $\geqslant 2d-6$ & $\geqslant 2d-6$ & \dots   \\ \hline
		$h'(k)$ & 1 & 3 & 5 & \dots & $2m+1$ & \dots & $2d-7$ & $2d-5$ & $2d-4$ & \dots \\ \hline
	\end{tabular}
	\caption{Values of Hilbert function of ideal $I$.}
	\label{table}
	\end{table}
	One can see that for fixed $d$ the sum of $h_I(k)$, $0\leqslant k\leqslant 2d-4$, is at least $2(d^2-2d-5)$ while the~sum of $h'(k)$ equals exactly $2(d-2)(d-1)$. So, one has 
		\begin{multline*}
		\sum\limits_{k=0}^{2d-4}\bigg(h_I(k)-h'(k)\bigg)\geqslant  \\
		\geqslant 2\sum\limits_{k=2}^{d-4}\bigg((2k+2)-(2k-1)\bigg)+ 2 \bigg( (2d-6) -(2d-5)\bigg) +\bigg((2d-6)-2(d-2)\bigg)=\\
		=2d-14,
	\end{multline*}
	which proves the lemma for $d>7$. 
	
	Now, suppose $d=7$. Then $2(d^2-2d-5)=2(d-2)(d-1)=60$ and therefore
	$$\sum\limits_{k=0}^{2d-4}h_I(k)\geqslant 1+3+6+8+\dots+8+6+3+1=60= 2(d-2)(d-1).$$ Here, the equality is attained if $I$ has $h$-vector $(1,3,6,8,8,8,8,8,6,3,1)$. 
	
	Let $d=6$. In this case, we have 
	$$\sum\limits_{k=0}^{2d-4}h_I(k)\geqslant 1+3+6+6+6+6+6+3+1= 38$$
	and $2(d-2)(d-1)=40$. Denote the sum by $S$. If $S=38$, then $h$-vector is $(1,3,6,6,6,6,6,3,1)$. \mbox{If $S=39$,} then due to the Gorenstein duality the middle value $h_I(d-2)$ is greater than the lower bound by 1. So, the $h$-vector is $(1,3,6,6,7,6,6,3,1)$. If $S=40$, then we have four possibilities to distribute additional 2 among values $h_I(k)=h_I(2d-4-k)$ and therefore $h$-vector is one of the following
	$$(1,4,6,6,6,6,6,4,1), \hspace{10pt}(1,3,7,6,6,6,7,3,1), \hspace{10pt} (1,3,6,7,6,7,6,3,1), \hspace{10pt} (1,3,6,6,8,6,6,3,1).$$
	For $S>40$ the inequality \eqref{inequality_for_singularities} holds.

\end{proof}
\begin{remark}
The sequence $h'(k)$ arises from the values of Hilbert function of complete intersection ideal of multidegree $(1,2,d-2,d-1)$.	
\end{remark}

\section{Nodal hypersurfaces of degree 6}\label{5}

Kloosterman's recent work \cite{kloosterman2022maximal} significantly advanced our understanding of the Ciliberto--Di Gennaro conjecture for hypersurfaces of high degree ($d\geqslant 7$). In this section, we build upon his approach and extend these results to the case of degree 6.

We use the notion of defect to detect the factoriality property.

\begin{Def}
	The \emph{defect} $\delta(X)$ of a nodal hypersurface $X\subset \IP^4$ is a number defined by the formula
\begin{equation*}
	h^{4}(X,\Q)-h^{2}(X,\Q).
\end{equation*}
\end{Def}

\begin{remark}[{\cite[Remark 3.5]{kloosterman2022maximal}}]\label{defect_and_factoriality}
Any nodal hypersurface $X$ satisfies the Poincar\'e duality property. In particular, there is an isomorphism $H^{4}(X,\Q)\simeq H_{2}(X,\Q)$ and therefore the Weil divisor class group~$\Cl(X)$ has rank equal to $h^{4}(X,\Q)$. On the other hand, the Lefschetz hyperplane theorem implies that the~Picard~group $\Pic(X)$ is isomorphic to $H^2(X,\Q)$. So, the defect $\delta(X)$ measures $\rk\,\Cl(X)/\Pic(X)$ and it is zero only for factorial varieties.
\end{remark}

The defect of a nodal hypersurface $X$ is deeply connected to the number of singularities.

\begin{Def}
		Let 
		$$X=\{F(x_0,\dots,x_4)=0\}\subset \IP^4=\Proj \,\C[x_0,x_1,x_2,x_3,x_4]$$
		be a hypersurface. Then one can define the \emph{Jacobian ideal} of $X$ as
		$$\Jac=\big(\frac{\partial}{\partial x_0}F,\dots, \frac{\partial}{\partial x_4}F\big)\subset \C[x_0,x_1,x_2,x_3,x_4].$$
\end{Def}
Let $J$ be the ideal of nodes, that is,
$$J=I(\Sing(X))=I(V(\Jac)).$$
Then $J$ is the saturation of $\Jac$ (see \cite[Exercise III-16]{eisenbud2006geometry} for definition). 

\begin{prop}[{\cite[Proposition 3.3]{dimca1990}}] \label{defect}
	Assume $X\subset \IP^4$ is a nodal hypersurface of degree $d$. Let $J$ be the ideal of nodes. Then
	\begin{equation*}
		\delta(X)=\#\Sing(X)-h_J(2d-5),
	\end{equation*}
	where $\#\Sing(X)$ is the number of singular points of $X$.
\end{prop}

Now, studying factoriality of nodal hypersurfaces we can switch from a geometric point of view to an algebraic one and investigate the behavior of Hilbert functions of zero-dimensional subschemes in~$\IP^4=\Proj(R)$.

If $X$ is not factorial, then by Remark \ref{defect_and_factoriality} it has a positive defect, i.e. $h_J(2d-5)<\#\Sing(X)$ due to Proposition \ref{defect}. Moreover, we have $h_J(2d-4)>0$. Otherwise, $h_J(k)=0$ for all $k\geqslant 2d-4$ due to Corollary \ref{macaulay_zero} and so $p_J(k)=0$. By taking a general hyperplane section of~$\Sing(X)$, one obtains an empty set in $\IP^3=\Proj(S)$. Assume the hyperplane is given by an equation $\{\ell=0\}$. Without loss of generality, we can choose $\ell$ equal to $x_4$. So, we get an artinian coordinate ring $R/(J,x_4)$, where $(J,x_4)$ is the~sum of $J$ and $(x_4)$. In particular, the ring $R/(J,\ell)$ is isomorphic to~$S/\overline{J}$, where $S=R/(x_4)\simeq\C[x_0,\dots,x_3]$ and $\overline{J}$ is the image of $J$ in~$S$. Then we construct an ideal $I\supset \overline{J}$ as follows. Let~$W$ be a codimension 1 subspace in $S_{2d-4}$ that contains~$J_{2d-4}$. Set $I=\bigoplus\limits_{k\geqslant 0} I_k$, where 
\begin{equation}\label{construction}\tag{5.1}
	I_k=\begin{cases}
		\{g\in S_k\,\mid \, g S_{2d-4-k}\subset W\}, \hspace{15pt}k\leqslant 2d-4;\\
		S_k, \hspace{123pt}k>2d-4.
	\end{cases}
\end{equation}
By construction, the quotient ring $S/I$ is local and the degree of any element $g\in S/I$ is at most~$2d-4$. Note that the socle of $S/I$ contains only polynomials of degree $2d-4$. Indeed, $g\in (S/I)_k$ for $k<2d-4$ cannot annihilate all images of $x_i$ in $S/I$. Otherwise, we get a contradiction since $V(I)$ is not just one point and $S/I$ is not local. Then the condition $\dim  (S/I)_{2d-4}= \text{codim } W=1$ implies that $\text{Soc} (S/I)$ is a 1-dimensional vector space over $\C$. So, it is generated by one element. Thus, the ring $S/I$ is Gorenstein of socle degree $2d-4$.

\begin{lem}[{\cite[Lemma 5.3]{kloosterman2022maximal}}]\label{contains_ci}
Suppose $d\geqslant 6$. Let $X \subset \IP^4$ be a nodal hypersurface of degree~$d$. Assume that $X$ has at most $2(d-2)(d-1)$ nodes and positive defect. Let $I$ be the ideal defined in~\emph{(\ref{construction})}.

If $h_I(d-4)\leqslant 2d-7$ then $\Sing(X)$ contains a subscheme which is a complete intersection of multidegree $(1,1,d-1,d-1)$ or of multidegree $(1,2,d-2,d-1)$.
\end{lem}

The following lemma was also proved by Kloosterman for hypersurfaces of degree $d\geqslant 7$. 
 
\begin{lem}\label{d=6}
	Suppose $d\geqslant 6$. Let $X \subset \IP^4$ be a nodal hypersurface of degree $d$ with positive defect. Let~$I$ be the ideal defined in \emph{(\ref{construction})}.
	
		If $h_I(d-4)>2d-7$ then $\Sing(X)$ consists of at least $2(d-1)(d-2)+1$ points.
\end{lem}

We split the proof of Lemma \ref{d=6} into two parts. The first part, due to Kloosterman, establishes certain restrictions on the values of the corresponding $h$-vector.  In the second part, we apply results on unimodality to eliminate most of the possible sequences with small values.  Then, using Macaulay's theorem for hyperplane sections, we obtain the necessary inequality for the number $\#\Sing(X)$ of nodes on $X$.

\begin{prop}\label{base_locus}
	Let $X\subset \IP^4$ be a non-factorial nodal hypersurface of degree $d$. Let $I$ be the ideal defined in \emph{(\ref{construction})}. Then the base locus $\Bs|I_{d-1}|$ is empty.
\end{prop}
\begin{proof}
	Let $\Jac$ be the Jacobian ideal of $X$. Then $\Jac$ is generated by polynomials of degree $d-1$. Note, that $\Jac\subset J\subset R$. Since $(\Jac,x_4)$ defines an empty set due to the choice of the hyperplane section. So, we have an inclusion $\Bs|(J,x_4)_{d-1}|\subset \Bs|(\Jac,x_4)_{d-1}|=\varnothing$. Therefore $\Bs|\overline{J}_{d-1}|$ is empty. Finally, since $I$ contains $\overline{J}$, the reverse inclusion for the base loci holds $\Bs|I_{d-1}|\subset \Bs|\overline{J}_{d-1}|=\varnothing$. 
\end{proof}

Now we can prove Lemma \ref{d=6}.
\begin{proof}[Proof of Lemma \ref{d=6}]
	According to Proposition \ref{defect}, the number of singular points on $X$ can be estimated through the Hilbert function, as
	\[ \#\Sing(X)=p_J(2d-4) \geqslant h_J(2d-4),\]
	where $J$ is the homogeneous ideal of $\Sing(X)$ and $p_J(\cdot)$ is the Hilbert polynomial. Using the exact sequence for a general hyperplane section ideal $\overline{J}$ as in the formula (\ref{short_exact_sequence}), we can express $h_J(2d-4)$ in terms of $h_{\overline{J}}$
	\begin{equation*}
		h_J(2d-4)=h_{\overline{J}}(2d-4)+h_J(2d-5)=\dots=\sum\limits_{k=0}^{2d-4}h_{\overline{J}}(k).
	\end{equation*}
	For  the homogeneous ideal $I\supset \overline{J}$ defined in (\ref{construction}) we have the following	
	\begin{equation*}
		\sum\limits_{k=0}^{2d-4}h_{\overline{J}}(k)\geqslant \sum\limits_{k=0}^{2d-4}h_I(k),
	\end{equation*}
	
Note that $h_I(1)\geqslant 3$ due to Lemma \ref{technical_lemma}. Indeed, let $d_k$ be the smallest integer number $t\geqslant 0$ such that the dimension of $\Bs|I_t|$ is less than $k$. If $h_I(1) < 3$, then $\sum d_k\leqslant 2d$ and equality holds for 
$$(d_2,d_1,d_0,d_{-1})=(1,1,d-1,d-1).$$
According to Remark \ref{ci_ideal_inside_ag_ideal}, there is a complete intersection ideal $I_{CI}\subset I$ of multidegree $(d_2,d_1,d_0,d_{-1})$. Corollary \ref{complete_intersection_coincides_with_ideal} implies that $I=I_{CI}$. However, these ideals have different Hilbert functions, as \mbox{$h_I(1)\geqslant 3$}, while $h_{I_{CI}}(1)=2$.

	Since $h_I(d-4)>2d-7$ by assumption and $\Bs|I_{d-1}|$ is empty by Proposition \ref{base_locus}, so due to Lemma~\ref{auxiliary_cii} one has 
\begin{equation}\label{inequality_final}
		\#\Sing(X)\geqslant \sum\limits_{k=0}^{2d-4} h_I(k)>2(d-2)(d-1)\tag{5.2}
	\end{equation}
	except when $d=7$ and $h$-vector is $(1,3,6,8,8,8,8,8,6,3,1)$ or $d=6$ and $h$-vector is one of the following
		$$(1,3,6,6,6,6,6,3,1),\hspace{10pt} v_1=(1,3,6,6,7,6,6,3,1), \hspace{10pt}v_2=(1,3,6,7,6,7,6,3,1), $$
		 $$v_3=(1,3,7,6,6,6,7,3,1), \hspace{10pt} (1,4,6,6,6,6,6,4,1), \hspace{10pt} v_4=(1,3,6,6,8,6,6,3,1).$$
However, $h$-vectors $v_1$, $v_2$, $v_3$, and $v_4$ never realize since $v_2$ and $v_3$ are not unimodal and $v_1$ and $v_4$ do not correspond to any artinian Gorenstein ring by Corollary \ref{136676631}.		 

So, it remains to exclude the $h$-vector $(1,3,6,8,8,8,8,8,6,3,1)$ for $d=7$ and $(1,3,6,6,6,6,6,3,1)$, $(1,4,6,6,6,6,6,4,1)$ for $d=6$.

	Again, by Remark \ref{ci_ideal_inside_ag_ideal} there is a complete intersection ideal $I_{CI}\subset I$ of multidegree $(d_2,d_1,d_0,d_{-1})$. The base locus $\Bs|I_{d-1}|$ is empty (see Proposition \ref{base_locus}). So, it implies $d_{-1}\leqslant d-1$.
	
	\emph{Case 1.} Consider $d=7$ and $h$-vector $(1, 3, 6, 8, 8, 8, 8, 8, 6, 3, 1)$. From Proposition \ref{new}, it follows that $\Bs\,|I_5|$ is empty and therefore $d_{-1}\leqslant 5$. Since $h_I(1)=3$, the ideal $I$ has only one generator $f$ of degree 1. Hence, the base locus of $|I_1|$ and of $|I_2|$ is a hyperplane and $d_2=1$. The base locus of $|I_3|$ has dimension at most 1, because $\dim I_3>\dim (f)_3$. Consequently, there is the following chain of inequalities
	 $$1=d_2 \leqslant  d_1\leqslant 3 \leqslant d_0 \leqslant d_{-1} \leqslant 5.$$
	 If $\sum\limits_{i=-1}^2 (d_i -1)$, i.e. the socle degree of $I_{CI}$, is less than $2d$, we are done by Lemma \ref{technical_lemma}. Otherwise, when $(d_2,d_1,d_0,d_{-1})=(1,3,5,5)$, the ideals $I_{CI}$ and $I$ must coincide due to their equal socle degree by Corollary \ref{complete_intersection_coincides_with_ideal}, contradicting their different Hilbert functions.
	 
	 \emph{Case 2.} Consider $d=6$ and $h$-vector $(1,3,6,6,6,6,6,3,1)$ or $(1,4,6,6,6,6,6,4,1)$. Proposition \ref{new} shows emptyness of $\Bs\,|I_3|$. So, we have  $d_2 \leqslant d_1 \leqslant d_0 \leqslant d_{-1} \leqslant 3$. Then we compute the socle degree of $I_{CI}$ which is
	 $$\sum\limits_{i=-1}^{2}(d_i -1)\leqslant 8 <2d$$
	  and get a contradiction.
	 
\end{proof} 

Our modification of Lemma \ref{d=6} makes Kloosterman's proof (see \cite[Theorem 5.5]{kloosterman2022maximal}) suitable for~$d=6$~too. Recall that the singular locus of $X$ is equipped with the scheme structure since it is defined by the Jacobian ideal (see \cite[Definition V-21]{eisenbud2006geometry} for details).

\begin{prop} \label{reduced}
	If $X\subset\IP^{n}$ is a nodal hypersurface, then the scheme $\Sing(X)$ is reduced.
\end{prop}
\begin{proof}
Let $U=\mathbb{A}^n$ be an affine chart containing $P\in \Sing(X)$. Then in this chart the hypersurface~$X$ is given by an equation $\bar{F}=0$, where $\bar{F}\in \C[y_1,\dots,y_n]$. Since $P$ is a node, we can find an analytic neighborhood of $P$ and local coordinates $\{z_i\}_{i=1,\dots, n}$ such that the equation $\bar{F}=0$ is 
	\begin{equation}\label{odt}
				\sum\limits_{i=1}^{n} z_i^2=0\tag{5.3}.
	\end{equation}
	By considering equation (\ref{odt}), we find all derivatives
	$$ \frac{\partial}{\partial z_j}\big(\sum\limits_{i=1}^n z_i^2\big)=2z_j$$
	and conclude the Hessian matrix
	\begin{equation*}
		\bigg( \frac{\partial^2}{\partial z_j z_k}\big(\sum\limits_{i=1}^n z_i^2\big) \bigg)
	\end{equation*}
	is invertible at $P$.
	
	Now, it remains to show that an intersection 
	$$\{\frac{\partial \bar{F}}{\partial y_1}=\dots= \frac{\partial \bar{F}}{\partial y_n}=0\}$$
	is transversal at $P$, i.e. the gradient vectors 
	$$(\frac{\partial^2 \bar{F}}{\partial y_1 \partial y_i}, \dots, \frac{\partial^2 \bar{F}}{\partial y_n \partial y_i})$$
	are linearly independent or, equivalently, the Hessian matrix $(\frac{\partial^2 \bar{F}}{\partial y_i \partial y_j})$ is non-degenerate at $P$. Note that
	\begin{equation*}
		\frac{\partial^2 \bar{F}}{\partial y_i \partial y_j}=\sum\limits_{k=1}^n\big( \sum\limits_{l=1}^n \frac{\partial^2 \bar{F}}{\partial z_l \partial z_k} \frac{\partial z_l}{\partial y_i} \frac{\partial z_k}{\partial y_j}+ \frac{\partial \bar{F}}{\partial z_k} \frac{\partial^2 z_k}{\partial y_i \partial y_j} \big).
	\end{equation*}
Since $P$ is singular, all $\frac{\partial \bar{F}}{\partial z_k}(P)$ vanish and we have 
$$\bigg(\frac{\partial^2 \bar{F}}{\partial y_i \partial y_j}(P)\bigg)=\bigg( \frac{\partial z_l}{\partial y_i} (P) \bigg)\bigg( \frac{\partial^2 \bar{F}}{\partial z_i \partial z_j}(P)\bigg) \bigg( \frac{\partial z_k}{\partial y_j}(P)\bigg).$$ 
Thus, the left hand side is a non-degenerate matrix as a product of invertible ones. So, each point \mbox{$P\in \Sing(X)$} is a transversal intersection of hypersurfaces and therefore $\Sing(X)$ is reduced.
\end{proof}

\begin{prop}\label{local_coordinates_and_nodes}
	Let $X$ be a nodal hypersurface in $\IP^n$ which is given by an equation $F=0$. Let $f_1,\dots,f_n$ be a regular sequence. Suppose the scheme $\Sing(X)$ contains a complete intersection 
	$$\Sigma= \{f_1=f_2=\dots= f_n=0\}$$
	 such that $$F=\sum\limits_{i=1}^n \varphi_i f_i$$ for some homogeneous polynomials $\varphi_i$. Then $\varphi_i(P)=0$ for all $P\in \Sigma$ and all $i=1,\dots,n$.
\end{prop}
\begin{proof}
	Consider an affine chart $U$ containing $P\in \Sigma$. Without loss of generality, we can assume that~$U=\{x_0\neq 0\}$. Denote by
$\bar{f_1}, \dots, \bar{f_n}$ the polynomials in $\C[y_1,\dots,y_n]$ defined as
$$\bar{f_i}(y_1,\dots,y_n)=f_i(1,y_1,\dots,y_n),$$
where $i=1,\dots,n$ and $y_1=\frac{x_1}{x_0},\dots, y_n=\frac{x_n}{x_0}$. Note that $\Sigma\subset \IP^n$ is a 0-dimentional subscheme due to the regularity of the sequence $f_1,\dots,f_n$. Thus, it remains 0-dimensional in the affine chart $U$.

First, we show that $\bar{f_1},\dots,\bar{f_n}$ form a local system of coordinates at some neighborhood of~$P\in \Sigma$. It~suffices to show that gradient vectors 
\begin{equation*}\label{gradient_vectors}\tag{5.4}
(\frac{\partial }{\partial y_1}\bar{f}_i, \hspace{5pt} \dots, \hspace{5pt} \frac{\partial }{\partial y_n}\bar{f}_i), \hspace{50pt}i=1,\dots,n,	
\end{equation*}
are linearly independent at $P$. Suppose the contrary. On the one hand, a matrix made up of vectors~\eqref{gradient_vectors} has rank less than $n$ at the point $P$. On the other hand, this matrix is the Jacobi matrix of the 0-dimensional subscheme $\bar{\Sigma}\subset \mathbb{A}^n$ given by $\bar{f_1},\dots, \bar{f_n}$. So, the point $P$ must be singular on $\bar{\Sigma}$. However, $\bar{\Sigma}$ is reduced as a subscheme of $\Sing(X)$ by Proposition \ref{reduced} and therefore it must be smooth. Thus we get a contradiction. Then $\bar{f_1}, \dots, \bar{f_n}$ are local coordinates.

Now, let $\bar{F}$ and $\bar{\varphi_i}$ the polynomials corresponding to $f$ and $\varphi_i$ in~$U$. Then the equation of $X$ in the chart $U$ can be written as
	$$\bar{F}=\sum\limits_{i=1}^n \bar{\varphi_i}\bar{f_i}=0.$$
	 Moreover, for its singular point $P\in \Sigma$ and local coordinates $\bar{f_1},\dots,\bar{f_n}$ at $P$ the following condition is satisfied 
	 \begin{equation*}\label{derivatives_vanish}
	 	\dfrac{\partial \bar{F}}{\partial \bar{f_1}}(P)=\dots=\dfrac{\partial \bar{F}}{\partial \bar{f_n}}(P)=0.\tag{5.5}
	 \end{equation*}
	 For every $i=1,\dots,n$, the calculations show
	 \begin{equation*}
	 	\dfrac{\partial \bar{F}}{\partial \bar{f_i}}(P)=\bar{\varphi_i}(P)+\dfrac{\partial \bar{\varphi_i}}{\partial \bar{f_i}}(P) \bar{f_i}(P).
	 \end{equation*}
	However, $\bar{f_i}(P)=0$ because $P\in X$ is a point where all $f_j$ vanish simultaneously. Then $\bar{\varphi_i}(P)$ must be equal to $0$ due to condition \eqref{derivatives_vanish} implying that $\varphi_i(P)=0$  for all $i$ and all $P\in \Sigma$. 
\end{proof}

\begin{proof}[Proof of Theorem \ref{main}]
	If $X$ is not factorial, then due to Lemma \ref{contains_ci} the subscheme $\Sing(X)$ contains a complete intersection $\Sigma$ of multidegree either $(1,1,d-1,d-1)$ or $(1,2,d-2,d-1)$. Consequently, we have $\#\Sing(X)\geqslant (d-1)^2$ if $\Sigma$ has multidegree $(1,1,d-1,d-1)$ or $\#\Sing(X)=2(d-2)(d-1)$ if $\Sigma$ has multidegree $(1,2,d-2,d-1)$.
	
	Let $f_1,\dots,f_4$ be a regular sequence which generates an ideal $I(\Sigma)$ defining $\Sigma$. Thus it is generated by all homogeneous polynomials vanishing on $\Sigma$. Let $F=0$ be the equation of $X$ in $\IP^4$. Therefore, there is an inclusion $F\in I(\Sigma)$ and, in particular, for some polynomials $\varphi_i$ one has 
	$$F=\sum\limits_{i=1}^4 \varphi_i f_i.$$
	
	By Proposition \ref{local_coordinates_and_nodes}, each $\varphi_i$ vanishes on $\Sigma$. Hence, all $\varphi_i$ are also elements of $I(\Sigma)$, i.e. 
	$$\varphi_i=\sum\limits_{j=1}^4 g_{ij} f_j$$
	for some polynomials $g_{ij}$.

	Now, suppose that $$\deg f_1 \leqslant \dots \leqslant \deg f_4,$$
	implying $\deg f_j >2$ for $j>2$ and $d-2 \leqslant \deg f_3 \leqslant \deg f_4$. To express $F$ of degree $d$ as 
	$$F=\sum\limits_{i=1}^4\sum\limits_{j=1}^i g_{ij}f_i f_j$$
	requires that	
	$$d=\deg F\geqslant \deg f_i +\deg f_j \geqslant 2 \deg f_j.$$
	Comparing with 
	$$2\deg f_3\geqslant 2d-4 >d=\deg F,$$
	 we conclude that $\deg f_j \leqslant 2$ and $j\leqslant 2$. 
	 
	 Therefore, $F\in (f_1,f_2)$, and there are two possibilities for~$X$. If~$ \deg f_1=\deg f_2 =1$, then $X$~contains a plane. If $\deg f_1=1$ and $\deg f_2=2$, then $X$ contains a quadric surface.
\end{proof}

\bibliographystyle{alpha}

\end{document}